\documentclass[11pt]{amsart}

\usepackage{amsmath,amssymb,amsthm,mathrsfs,xcolor}
\usepackage[margin=1.1in]{geometry}
\usepackage{hyperref}
\usepackage{mathtools}
\mathtoolsset{showonlyrefs}
\usepackage{zref-clever}
\newcommand{\ii}{\mathrm{i}}
\newcommand{\R}{\mathbb{R}}
\newcommand{\C}{\mathbb{C}}
\newcommand{\la}{\lambda}
\newcommand{\wh}{\widehat}

\newcommand{\Lone}{L^1}

\newcommand{\dx}{\,dx}

\newcommand{\ds}{\,ds}

\newcommand{\deta}{\,d\eta}
\newcommand{\dsig}{\,d\sigma}

\newcommand{\Linfty}{L^\infty}
\newcommand{\Ltwo}{L^2}
\newcommand{\Hone}{H^1}
\newcommand{\Htwo}{H^2}
\usepackage{comment}
\numberwithin{equation}{section}

\theoremstyle{plain}
\newtheorem{theorem}{Theorem}[section]
\newtheorem{lemma}[theorem]{Lemma}
\newtheorem{proposition}[theorem]{Proposition}

\theoremstyle{definition}

\usepackage{scalerel,stackengine}
\stackMath
\newcommand\reallywidehat[1]{%
\savestack{\tmpbox}{\stretchto{%
  \scaleto{%
    \scalerel*[\widthof{\ensuremath{#1}}]{\kern-.6pt\bigwedge\kern-.6pt}%
    {\rule[-\textheight/2]{1ex}{\textheight}}
  }{\textheight}%
}{0.5ex}}%
\stackon[1pt]{#1}{\tmpbox}%
}
\begin{document}

\title[Final-state problem for 1D cubic NLS]{On the final-state problem for the 1D cubic NLS}

\author{Gong Chen}

\author{Yongyu Qiang}

\address{School of Mathematics, Georgia Institute of Technology. 686 Cherry Street,
Atlanta, GA }
\email{gc@math.gatech.edu}
\email{yqiang7@gatech.edu}
\date{\today}

\thanks{
G.C. was partially supported by NSF Grants DMS-2350301 and CAREER-DMS-2540992, by the Simons Foundation MP-TSM-00002258, and by the Stefan Bergman Fellowship. We also thank the support from the REU program in the School of Mathematics, supported by the Georgia Tech College of Sciences and NSF Grant DMS-2244427. Finally, we thank Mahnav Petersen, who participated in the 2024 REU program, for helpful earlier discussions.
}
\begin{abstract}
We consider the one-dimensional cubic nonlinear Schr\"odinger equation
$$
\ii\partial_tu+\frac12\partial_{xx}u=\la|u|^2u,\,\lambda=\pm 1
$$
and solve the final-state (modified wave operator) problem for small asymptotic  data. More precisely, given a small $W(\xi)$, we construct  a solution $u$ such that 
\begin{equation*}
    u\rightarrow (2\pi)^{-1/2}(\ii t)^{-1/2}e^{\ii x^2/(2t)}\,
W\!\Big(\frac{x}{t}\Big)\exp(-\ii\la|W(x/t)|^2\log t).
\end{equation*}
Crucially, we design  a contraction map, so that we can run the analysis in the spirit of  Kato--Pusateri \cite{KP} for $w$ with a forcing term depending {\it only} on the final data $W$. This scheme is easy to adapt to solving final state problems with a complete theory for the forward problems.

\end{abstract}
\maketitle

\section{Introduction}\label{sec:intro}

\subsection{Background}
We study the cubic NLS on $\R$,
\begin{equation}\label{eq:NLS}
\ii\partial_tu+\frac12\partial_{xx}u=\la|u|^2u,\,u(0)=u_0,\qquad (t,x)\in\R\times\R,\qquad \la=\pm 1.
\end{equation}
First, recall that sufficiently regular solutions of \eqref{eq:NLS} conserve the $L^{2}$-norm
\[M(u) := \int\left|u\right|^{2}\,dx \]
and the total energy (Hamiltonian):
\begin{align}\label{Ham}
H\left(u\right) := \int\frac{1}{2}\left|\partial_xu\right|^{2}
  \pm\frac{1}{4}\left|u\right|^{4} \, dx.
\end{align}
The Cauchy problem for \eqref{eq:NLS} 
is globally well-posed in $L^2$ and $H^1$, see for example Cazenave-Weissler \cite{CW}.

We are interested in global-in-time bounds and asymptotics
for \eqref{eq:NLS} as $|t|\rightarrow \infty$.
The main feature of the cubic nonlinearity is its criticality with respect to scattering:
linear solutions of the Schr\"odinger equation decay at best like $|t|^{-1/2}$ in $L^\infty_x$, 
so that, when evaluating the nonlinearity on linear solutions, one see that $|u|^{2}u \sim |t|^{-1}u$;
the non-integrability of $|t|^{-1}$ 
results in a ``Coulomb''-type contribution of the nonlinear terms which produces 
modified scattering with an additional nonlinear phase correction compared to the linear behavior. More precisely, for reasonable initial data, as time goes to infinity, the  solution to \eqref{eq:NLS} behaves like $(2\pi)^{-1/2}(\ii t)^{-1/2}e^{\ii x^2/(2t)}\,
W\!\Big(\frac{x}{t}\Big)\exp(-\ii\la|W(x/t)|^2\log t)$ for some $W$.

To solve the initial value problems, using complete integrability, modified scattering was proven in the seminal work of Deift-Zhou \cite{DZ} 
without size restriction on the solutions.
Without making use of complete integrability, and
restricting the analysis to small solutions, 
proofs of modified scattering were given by Ozawa \cite{O},
Hayashi-Naumkin \cite{HN}, Lindblad-Soffer \cite{LS}, Kato-Pusateri \cite{KP}, and Ifrim-Tataru \cite{IT}. We refer to the survey by Murphy  \cite{MurphyReview} for more references. 
%

In this paper, we are interested in the opposite direction. Given $W$ as the modified scattering profile, we want to find a solution to \eqref{eq:NLS} achieving this given modified scattering behavior.  This is the so-called final state problem and  the existence of modified wave operators, we refer to Hayashi-Naumkin \cite{HN2}, Ifrim-Tataru \cite{IT}, Kawamoto-Mizutani \cite{KM},  Segata \cite{Se}, and references therein for more details and history. The purpose of this paper is to revisit the final state problem from the viewpoint of space-time resonance computations in the spirit of Kato-Pusateri \cite{KP}. Our purpose is to propose a scheme which can be easily adapted to other settings with a complete forward scattering analysis. One can easily see that our analysis can be easily applied to the  Hartree equation
\begin{equation*}
    \ii\partial_t u+\frac{1}{2}\Delta u=\big(|x|^{-1}     \ast |u|^2\big) u, \qquad (t,x)\in \mathbb{R}\times \mathbb{R}^n,\,n\geq2
\end{equation*}
as in Kato-Pusateri \cite{KP}. We also expect our analysis can be adapted to the perturbed settings like the models considered  in  Chen-Pusateri \cite{CP}, Germain-Pusateri-Rousset \cite{GPR}, Segata \cite{Se} etc.
\subsection{Main theorem}
Before introducing our main results, it is necessary to introduce some notations.

Define the linear Schr\"odinger propagator as
$$
U(t)=e^{\ii t\partial_{xx}/2}.
$$
We use the Fourier transform
$$
\wh f(\xi)=\int_{\R}e^{-\ii x\xi}f(x)\,dx,
\qquad
f(x)=\frac1{2\pi}\int_{\R}e^{\ii x\xi}\wh f(\xi)\,d\xi.
$$
With this convention, $U(t)$ acts in Fourier space by
$$
\reallywidehat{U(t)f}(\xi)=e^{-\ii t\xi^2/2}\wh f(\xi).
$$
Define the profile
$$
f(t)=U(-t)u(t).
$$
\begin{theorem}\label{thm:main}
We fix parameters
\begin{equation}\label{eq:params}
0<\delta<\frac14,\qquad 0<\alpha<\delta,
\end{equation}
and assume small final data\footnote{One can see that from our proof that the $H^2_\xi$ can be weakened to $H_\xi^{3/2+}$. }
\begin{equation}\label{eq:Wsmall}
W\in \Linfty_\xi\cap \Htwo_\xi,
\qquad
\|W\|_{\Linfty_\xi}+\|W\|_{\Htwo_\xi}\le\varepsilon_0,
\end{equation}
for $\varepsilon_0>0$ sufficiently small.

Define the explicit asymptotic profile
\begin{equation}\label{eq:vdef}
v(t,\xi)=W(\xi)e^{-\ii\la|W(\xi)|^2\log t}.
\end{equation}
Assume \eqref{eq:params} and \eqref{eq:Wsmall}, and let $v$ be defined by \eqref{eq:vdef}.
Then there exists $T\ge 2$ and a unique solution $u\in C([T,\infty);\Ltwo_x)$ of \eqref{eq:NLS}
such that with $f(t)=U(-t)u(t)$,
\begin{equation}\label{eq:mainbound}
\sup_{t\ge T}t^\alpha\Big(\|\wh f(t)-v(t)\|_{\Linfty_\xi}+\|\wh f(t)-v(t)\|_{L^2_\xi}+(1+\log t)^{-1}\|\partial_\xi(\wh f(t)-v(t))\|_{\Ltwo_\xi}\Big)<\infty.
\end{equation}
Moreover, for all $t\ge T$ we have the expansion
\begin{equation}\label{eq:u-asymp-main}
u(t,x)=\frac{1}{\sqrt{2\pi}}\frac{e^{\ii x^2/(2t)}}{\sqrt{\ii t}}\,
W\!\Big(\frac{x}{t}\Big)\exp\!\Big(-\ii\la\Big|W\!\Big(\frac{x}{t}\Big)\Big|^2\log t\Big)
+\mathrm{Err}(t,x),
\end{equation}
where the remainder satisfies the pointwise bound
\begin{equation}\label{eq:u-asymp-err}
\|\mathrm{Err}(t)\|_{L^\infty_x}
\le C\,t^{-1/2-\alpha}+C\,t^{-3/4}\,\varepsilon_0(1+\log t),
\qquad t\ge T.
\end{equation}
\end{theorem}

\subsection{Idea of the proof}
Using \eqref{eq:vdef}, we build an approximate solution $u_{\mathrm{app}}$ that already contains the long-range phase, and solve for a correction by a backward fixed point.

More precisely, schematically, the problem is reduced to find the fixed point to the following map
\begin{equation}
\Phi(g)(t)=\ii\la\int_t^\infty U(-s)\Big(|u|^2u-|u_{\mathrm{app}}|^2u_{\mathrm{app}}\Big)(s)\,\ds
-\ii\int_t^\infty U(-s)\varepsilon(s)\,\ds,
\end{equation}
with $u=u_{\mathrm{app}}+U(\cdot)g$ and $w=U(\cdot)g$.
The stationary phase is used only to estimate the forcing created by the approximation. All nonlinear correction terms are estimated by
dispersive bounds and careful Fourier-side handling of $\partial_\xi$ in the style of \cite{KP}.  We will see that $\varepsilon$ depends only on $W$, which is given (independent of the nonlinear solution). So the final state analysis is reduced to applying the stationary phase analysis to this term, given that the forward scattering analysis is already established.

\section{Preliminaries}\label{sec:tools}
We first recall some  notations, basic inequalities, and dispersive estimates which will be applied in this paper. 
The following dispersive estimate is standard from the stationary phase. For a proof, see for example \cite{GPR}.
\begin{lemma}\label{lem:KPdisp}
There exists $C>0$ such that for all $t\neq0$ and all $h\in \Ltwo_x$ with
$\wh h\in \Linfty_\xi$ and $\partial_\xi\wh h\in \Ltwo_\xi$,
\begin{equation}\label{eq:KPdisp}
\|U(t)h\|_{\Linfty_x}\le C\Big(|t|^{-1/2}\|\wh h\|_{\Linfty_\xi}+|t|^{-3/4}\|\partial_\xi\wh h\|_{\Ltwo_\xi}\Big).
\end{equation}
Moreover, one has
\begin{equation}\label{eq:SP-template}
\Big\|U(t)h-(2\pi)^{-1/2}(\ii t)^{-1/2}e^{\ii x^2/(2t)}\,\wh h\!\Big(\frac{x}{t}\Big)\Big\|_{L^\infty_x}
\le C\,t^{-3/4}\|\partial_\xi \wh h\|_{L^2_\xi}.
\end{equation}
\end{lemma}
To prepare the trilinear analysis later on, we introduce the notation
\begin{equation}
    \mathcal{T}(f_1,f_2,f_3)(s):= \ii U(-s)\Big((U(s)f_1)(\overline{U(s)f_2})( U(s)f_3)\Big).
\end{equation}
We record some computations following Kato-Pusateri \cite{KP}. 
Taking the Fourier transform, it follows that
\begin{equation}
    \widehat{\mathcal{T}}(f_1,f_2,f_3)(s)= (2\pi)^{-1}\ii
 \iint
e^{is\eta(\xi-\sigma)}
\widehat f_1(\xi-\eta)\,
\overline{\widehat f_2(\sigma-\eta)}\,
\widehat f_3(\sigma)
\, d\eta d\sigma
\end{equation}
since
\[
\frac12\big[\xi^2-(\xi-\eta)^2+(\sigma-\eta)^2-\sigma^2\big]
= \eta(\xi-\sigma).
\]
Changing variable in the \(\sigma\)-integral, with the Plancherel theorem, it can also be written as 
\begin{align}
    \widehat{\mathcal{T}}(f_1,f_2,f_3)(s)
&= (2\pi)^{-1}\ii
\iint
e^{is\eta\sigma}
\widehat f_1(\xi-\eta)\,
\overline{\widehat f_2(\xi-\sigma-\eta)}\,
\widehat f_3(\xi-\sigma)
\, d\eta d\sigma\\
&=
(2\pi)^{-1}\ii
 \iint
\mathcal F_{\eta,\sigma}\!\left[e^{is\eta\sigma}\right](\eta',\sigma')
\mathcal F^{-1}_{\eta,\sigma}[F](\eta',\sigma',\xi)
\, d\eta' d\sigma'  \\
&= (2\pi)^{-1}\ii
\frac{1}{s}\iint
 e^{-i\eta'\sigma'/s}
\mathcal F^{-1}_{\eta,\sigma}[F](s,\eta',\sigma',\xi)
\, d\eta' d\sigma',
\end{align}
where
\[
F(\eta,\sigma,\xi)
:= \widehat f_1(\xi-\eta)\,
\overline{\widehat f_2(\xi-\sigma-\eta)}\,
\widehat f_3(\xi-\sigma).
\]
Then one has
\begin{align}
      \widehat{\mathcal{T}}(f_1,f_2,f_3)(s)&=
(2\pi)^{-1}\ii
 \frac{1}{s}
\iint \mathcal F^{-1}_{\eta,\sigma}[F](\eta',\sigma',\xi)
\, d\eta' d\sigma'  \\
&\quad
+(2\pi)^{-1}\ii
 \frac{1}{s}
\iint \big(e^{-i\eta'\sigma'/s}-1\big)
\mathcal F^{-1}_{\eta,\sigma}[F](\eta',\sigma',\xi)
\, d\eta' d\sigma'  \\
&=
 \ii \frac{1}{s}
\widehat f_1(\xi)\overline{\widehat f_2(\xi)}\widehat f_3(\xi)
+  R(f_1,f_2,f_3)(s,\xi),
\end{align}
where the remainder is
\begin{equation}
    R(f_1,f_2,f_3)(s,\xi):=(2\pi)^{-1}\ii
 \frac{1}{s}
\iint \big(e^{-i\eta'\sigma'/s}-1\big)
\mathcal F^{-1}_{\eta,\sigma}[F](\eta',\sigma',\xi)
\, d\eta' d\sigma'. 
\end{equation}
From the computations of Kato-Pusateri \cite{KP}, we have the following estimates:
\begin{lemma}\label{lem:trilinear}Let $0<\delta<\frac{1}{4}$. One has
\begin{equation}
        \widehat{\mathcal{T}}(f_1,f_2,f_3)(s)= \ii \frac{1}{s}
\widehat f_1(\xi)\overline{\widehat f_2(\xi)}\widehat f_3(\xi)
+  R(f_1,f_2,f_3)(s,\xi)
\end{equation}
and
\begin{equation}
    | R(f_1,f_2,f_3)(s,\xi)|
\lesssim
s^{-1-\delta}
\|f_1\|_{H_x^{0,1}}\|f_2\|_{H_x^{0,1}}\|f_3\|_{H_x^{0,1}}.
\end{equation}
\end{lemma}
\begin{proof}
    See \cite[Section 2]{KP} and a similar proof for Lemma \ref{lem:Rlinf} later on.
\end{proof}
\section{Approximate solution and forcing identity}\label{sec:approx}

\subsection{The approximate profile and approximate solution}
In this section, we define the approximate solution given $W$ and its properties.

Given $W$, we define $v$ as in \eqref{eq:vdef}, then
\begin{equation}\label{eq:vODE}
\partial_tv(t,\xi)=-\frac{\ii\la}{t}|v(t,\xi)|^2v(t,\xi).
\end{equation}
Let $\varphi(t)$ be defined by
\begin{equation}\label{eq:phi-def}
\wh\varphi(t,\xi)=v(t,\xi).
\end{equation}
Define the approximate solution
\begin{equation}\label{eq:uapp-def}
u_{\mathrm{app}}(t)=U(t)\varphi(t).
\end{equation}
Define the forcing error
\begin{equation}\label{eq:eps-def}
\varepsilon(t):=\ii U(t)\partial_t\varphi(t)-\la|u_{\mathrm{app}}(t)|^2u_{\mathrm{app}}(t),
\end{equation}
so that $u_{\mathrm{app}}$ solves the forced equation
\begin{equation}\label{eq:uapp-forced}
\ii\partial_tu_{\mathrm{app}}+\frac12\partial_{xx}u_{\mathrm{app}}
=\la|u_{\mathrm{app}}|^2u_{\mathrm{app}}+\varepsilon.
\end{equation}
For a function $h=h(\xi)$ define
\begin{equation}\label{eq:Rdef}
R[h](t,\xi)=\frac{1}{2\pi t}\iint_{\R^2}\Big(e^{-\ii\eta'\sigma'/t}-1\Big)\mathcal{F}^{-1}_{\eta,\sigma}[F_h](t,\eta',\sigma',\xi)\,\deta'\,\dsig',
\end{equation}
where
\begin{equation}
    F_h(\eta,\sigma,\xi)=
h(t,\xi-\sigma)\,h(t,\xi+\eta)\,\overline{h(t,\xi+\eta-\sigma)}.
\end{equation}
%
With notations above, the key identity is the following:
\begin{lemma}\label{lem:forcing}
With $\varepsilon$ defined by \eqref{eq:eps-def} and $v$ by \eqref{eq:vdef},
\begin{equation}\label{eq:forcing-id}
\reallywidehat{U(-t)\varepsilon}(t,\xi)=-\ii\lambda R[v](t,\xi).
\end{equation}
\end{lemma}

\begin{proof}
Apply $U(-t)$ to \eqref{eq:eps-def}:
$$
U(-t)\varepsilon=\ii\partial_t\varphi-\la\,U(-t)\big(|u_{\mathrm{app}}|^2u_{\mathrm{app}}\big).
$$
Taking Fourier transforms and using $\wh\varphi=v$ gives
$$
\reallywidehat{U(-t)\varepsilon}(t,\xi)=\ii\partial_tv(t,\xi)-\la\,\reallywidehat{U(-t)\big(|U(t)\varphi|^2U(t)\varphi\big)}(\xi).
$$
A standard computation, Lemma \ref{lem:trilinear},  (also see \cite{KP}) shows
$$
\reallywidehat{U(-t)\big(|U(t)\varphi|^2U(t)\varphi\big)}(\xi)=\frac1t|v(t,\xi)|^2v(t,\xi)+ \ii R[v](t,\xi),
$$
where $R[v]$ is exactly \eqref{eq:Rdef}.
Using the ODE \eqref{eq:vODE}, we have $\ii\partial_tv=(\la/t)|v|^2v$, hence \eqref{eq:forcing-id}.
\end{proof}


\begin{proposition}\label{prop:uapp-Linf}
Let $u_{\mathrm{app}}$ be defined by \eqref{eq:uapp-def}. Then for $s\geq 1$, one has
\begin{equation}\label{eq:uappLinf}
\|u_{\mathrm{app}}(s)\|_{\Linfty_x}\le C\varepsilon_0 s^{-1/2}+C\varepsilon_0(1+\log s)s^{-3/4}.
\end{equation}
\end{proposition}

\begin{proof}
Apply Lemma \ref{lem:KPdisp} to $h=\varphi(s)$.
Since $\wh\varphi=v(s)$, we get
$$
\|u_{\mathrm{app}}(s)\|_\infty=\|U(s)\varphi(s)\|_\infty
\le C\Big(s^{-1/2}\|v(s)\|_{\Linfty_\xi}+s^{-3/4}\|\partial_\xi v(s)\|_{\Ltwo_\xi}\Big).
$$
Now $\|v(s)\|_{\Linfty_\xi}=\|W\|_{\Linfty_\xi}\le\varepsilon_0$.
Also $\partial_\xi v(s,\xi)$ contains a term proportional to $(\log s)$, hence
$$
\|\partial_\xi v(s)\|_2\le C\varepsilon_0(1+\log s).
$$
This gives \eqref{eq:uappLinf}. 
\end{proof}


\section{Function space and the backward fixed point}\label{sec:fixedpoint}

Define the correction
$$
w(t)=u(t)-u_{\mathrm{app}}(t),
\qquad
g(t)=U(-t)w(t).
$$
Then
$$
u(t)=u_{\mathrm{app}}(t)+w(t),
\qquad
f(t)=U(-t)u(t)=\varphi(t)+g(t),
\qquad
\wh f(t,\xi)=v(t,\xi)+\wh g(t,\xi).
$$
Subtract \eqref{eq:uapp-forced} from the exact NLS for $u$ to get
$$
\ii\partial_tw+\frac12\partial_{xx}w=\la\Big(|u|^2u-|u_{\mathrm{app}}|^2u_{\mathrm{app}}\Big)-\varepsilon.
$$
Applying $U(-t)$ and using unitarity of $U(t)$ on $\Ltwo$ gives
\begin{equation}\label{eq:g-eq}
\partial_tg(t)=-\ii\la\,U(-t)\Big(|u|^2u-|u_{\mathrm{app}}|^2u_{\mathrm{app}}\Big)(t)+\ii\,U(-t)\varepsilon(t).
\end{equation}
We solve backward with terminal condition $g(t)\to0$ as $t\to\infty$:
\begin{equation}\label{eq:backward}
g(t)=\ii\la\int_t^\infty U(-s)\Big(|u|^2u-|u_{\mathrm{app}}|^2u_{\mathrm{app}}\Big)(s)\,\ds
-\ii\int_t^\infty U(-s)\varepsilon(s)\,\ds.
\end{equation}
For $T\ge2$ define
\begin{equation}\label{eq:XTdef}
\|g\|_{X_T}=\sup_{t\ge T}t^\alpha\Big(\|\wh g(t)\|_{\Linfty_\xi}+\|\wh g(t)\|_{L^2_\xi}+(1+\log t)^{-1}\|\partial_\xi\wh g(t)\|_{\Ltwo_\xi}\Big).
\end{equation}
Let
$$
B_M=\{g\in X_T:\ \|g\|_{X_T}\le M\}.
$$

\subsection{The fixed point map}
Define $\Phi$ by the right-hand side of \eqref{eq:backward}:
\begin{equation}\label{eq:PhiDef}
\Phi(g)(t)=\ii\la\int_t^\infty U(-s)\Big(|u|^2u-|u_{\mathrm{app}}|^2u_{\mathrm{app}}\Big)(s)\,\ds
-\ii\int_t^\infty U(-s)\varepsilon(s)\,\ds,
\end{equation}
with $u=u_{\mathrm{app}}+U(\cdot)g$ and $w=U(\cdot)g$.

We write
\begin{equation}\label{eq:Phi}
\Phi(g)=:\Phi_{\mathrm{nl}}(g)+\Phi_\varepsilon,
\end{equation}
where $\Phi_\varepsilon$ is independent of $g$.

\section{Analysis of the forcing term}\label{sec:forcing}

In this section, we study  bounds for $R[v]$ and $\partial_\xi R[v]$.
These estimates are used \emph{only} for the forcing term $\Phi_\varepsilon$ through Lemma \ref{lem:forcing}.

\begin{lemma}\label{lem:Rlinf}
Fix $0<\delta<1/4$. Let $h=h(\xi)$ and let $f=\mathcal{F}^{-1}h$.
Then for all $t\ge1$,
\begin{equation}\label{eq:Rlinf}
\|R[h](t)\|_{\Linfty_\xi}\le C t^{-1-\delta}\Big(\|f\|_{\Ltwo_x}+\|xf\|_{\Ltwo_x}\Big)^3
= C t^{-1-\delta}\|h\|_{\Hone_\xi}^3.
\end{equation}
\end{lemma}

\begin{proof}
Start from \eqref{eq:Rdef}. For $0<\delta<1$,
 $
|e^{-\ii z}-1|\le C_\delta |z|^\delta.
$
With $z=\eta'\sigma'/t$ we obtain
$$
|e^{-\ii\eta'\sigma'/t}-1|\le C_\delta t^{-\delta}|\eta'|^\delta|\sigma'|^\delta.
$$
Therefore
\begin{equation}\label{eq:Rlinf1}
|R[h](t,\xi)|
\le C\,t^{-1-\delta}\iint |\eta'|^\delta|\sigma'|^\delta
\Big|\mathcal F^{-1}_{\eta,\sigma}[F_h](t,\eta',\sigma',\xi)\Big|\,\deta'\,\dsig'.
\end{equation}
A direct Fourier inversion in $(\eta,\sigma)$ gives the physical-kernel identity
\begin{equation}\label{eq:physkernel}
\mathcal F^{-1}_{\eta,\sigma}[F_h](t,\eta',\sigma',\xi)
=(2\pi)^{1/2}e^{\ii\xi(\eta'+\sigma')}\int_\R e^{-\ii x\xi}\,
f(x-\eta')\,f(x)\,\overline{f(x-\sigma')}\,\dx.
\end{equation}
Taking absolute values and using $|e^{\ii\xi(\eta'+\sigma')}|=1$ yields
$$
\Big|\mathcal F^{-1}_{\eta,\sigma}[F_h](t,\eta',\sigma',\xi)\Big|
\le C\int_\R |f(x-\eta')|\,|f(x)|\,|f(x-\sigma')|\,\dx.
$$
Plugging this into \eqref{eq:Rlinf1} gives
\begin{equation}\label{eq:Rlinf2}
|R[h](t,\xi)|
\le C\,t^{-1-\delta}\iiint |\eta'|^\delta|\sigma'|^\delta
|f(x-\eta')|\,|f(x)|\,|f(x-\sigma')|\,\dx\deta'\dsig'.
\end{equation}
We now estimate the $\eta'$ and $\sigma'$ integrals pointwise in $x$.
Use the elementary inequality
$$
|\eta'|^\delta \le C\Big(|x-\eta'|^\delta+|x|^\delta\Big).
$$
Changing variables $y=x-\eta'$ yields
\begin{align*}
\int_\R |\eta'|^\delta |f(x-\eta')|\,d\eta'
&\le C\int_\R |x-\eta'|^\delta |f(x-\eta')|\,d\eta' + C|x|^\delta\int_\R|f(x-\eta')|\,d\eta'\\
&= C\int_\R |y|^\delta |f(y)|\,dy + C|x|^\delta\int_\R |f(y)|\,dy\\
&= C\||y|^\delta f\|_{\Lone_y}+C|x|^\delta\|f\|_{\Lone_y}.
\end{align*}
The same bound holds for the $\sigma'$-integral. Hence \eqref{eq:Rlinf2} implies
$$
|R[h](t,\xi)|
\le C\,t^{-1-\delta}\int_\R |f(x)|
\Big(\||y|^\delta f\|_{\Lone_y}+|x|^\delta\|f\|_{\Lone_y}\Big)^2\,dx.
$$
Expanding the square yields terms involving
$$
\|f\|_{\Lone_x},\qquad \||x|^\delta f\|_{\Lone_x},\qquad \||x|^{2\delta} f\|_{\Lone_x}.
$$
Since $0<\delta<1/4$,
%
\begin{equation}\label{eq:L1moment}
\|f\|_{\Lone_x}+ \||x|^\delta f\|_{\Lone_x}+\||x|^{2\delta} f\|_{\Lone_x}
\lesssim \|f\|_{\Ltwo_x}+\|xf\|_{\Ltwo_x}.
\end{equation}
Combining these bounds gives
$$
\|R[h](t)\|_{\Linfty_\xi}\le C\,t^{-1-\delta}\Big(\|f\|_{\Ltwo_x}+\|xf\|_{\Ltwo_x}\Big)^3.
$$
Finally, by Plancherel with our Fourier convention,
$$
\|f\|_{\Ltwo_x}=(2\pi)^{-1/2}\|h\|_{\Ltwo_\xi},
\qquad
\|xf\|_{\Ltwo_x}=(2\pi)^{-1/2}\|\partial_\xi h\|_{\Ltwo_\xi},
$$
so \eqref{eq:Rlinf} follows.
\end{proof}

\begin{lemma}\label{lem:dR}
Fix $0<\delta<1/4$. Let $h=h(\xi)$ and $f=\mathcal{F}^{-1}h$.
Then for all $t\ge1$,
\begin{equation}\label{eq:dR}
\|\partial_\xi R[h](t)\|_{\Ltwo_\xi}\le C\,t^{-1-\delta}\,\|h\|_{\Htwo_\xi}^3.
\end{equation}
\end{lemma}

\begin{proof}
Differentiate \eqref{eq:Rdef} under the integral:
\begin{equation}\label{eq:dR1}
\partial_\xi R[h](t,\xi)
= \frac{1}{2\pi t}\iint (e^{-\ii\eta'\sigma'/t}-1)\,
\partial_\xi\mathcal F^{-1}_{\eta,\sigma}[F_h](t,\eta',\sigma',\xi)\,\deta'\dsig'.
\end{equation}
From the inverse Fourier formula, one has
\begin{equation}\label{eq:dF-inv}
\partial_\xi \mathcal F^{-1}_{\eta,\sigma}[F_h](t,\eta,\sigma,\xi)
= i(\eta+\sigma)\,\mathcal F^{-1}_{\eta,\sigma}[F_h](t,\eta,\sigma,\xi)
+ \mathcal F^{-1}_{\eta,\sigma}[\partial_\xi F_h](t,\eta,\sigma,\xi),
\end{equation}
because $\partial_\xi$ hits both the unimodular factor $e^{i\xi(\eta+\sigma)}$ and the $\xi$-dependence inside $F_h$.
Hence
\begin{align}
\partial_\xi R[h](t,\xi)
&= \frac{1}{2 \pi t}\iint_{\mathbb R^2} \big(e^{-i\eta\sigma/t}-1\big)\,
\Big(i(\eta+\sigma)\,\mathcal F^{-1}_{\eta,\sigma}[F_h]
+ \mathcal F^{-1}_{\eta,\sigma}[\partial_\xi F_h]\Big)(t,\eta,\sigma,\xi)\,d\eta\,d\sigma.
\end{align}
So it suffices to estimate the two pieces:
\begin{align}
\mathcal I_1
&:= \frac{1}{2\pi t}\iint_{\mathbb R^2} \big(e^{-i\eta\sigma/t}-1\big)\,
(\eta+\sigma)\,\mathcal F^{-1}_{\eta,\sigma}[F_h](t,\eta,\sigma,\xi)\,d\eta\,d\sigma, \\
\mathcal I_2
&:= \frac{1}{2\pi t}\iint_{\mathbb R^2} \big(e^{-i\eta\sigma/t}-1\big)\,
\mathcal F^{-1}_{\eta,\sigma}[\partial_\xi F_h](t,\eta,\sigma,\xi)\,d\eta\,d\sigma.
\end{align}
Using $
\big|e^{-i\eta\sigma/t}-1\big|
\le t^{-\delta}
\Big(|\eta|^\delta\langle\eta\rangle^{-\frac{1}{2}-2\delta}\Big)
\Big(|\sigma|^\delta\langle\sigma\rangle^{-\frac{1}{2}-2\delta}\Big)
\langle\eta\rangle^{\frac{1}{2}+2\delta}\langle\sigma\rangle^{\frac{1}{2}+2\delta}$,  for any $\Ltwo_\xi$-valued function $G(\eta,\sigma,\xi)$,
\begin{equation}\label{eq:CS-etasigma}
\Big\|
\iint (e^{-i\eta\sigma/t}-1)\,G(\eta,\sigma,\cdot)\,d\eta\,d\sigma
\Big\|_{\Ltwo_\xi}
\;\le\;
t^{-\delta}\,\|w_\delta\|_{\Ltwo_{\eta,\sigma}}\;
\big\|\langle\eta\rangle^{\frac{1}{2}+2\delta}\langle\sigma\rangle^{\frac{1}{2}+2\delta}G\big\|_{\Ltwo_{\eta,\sigma}\Ltwo_\xi},
\end{equation}
where
$
w_\delta(\eta,\sigma)
:= \Big(|\eta|^\delta\langle\eta\rangle^{-\frac{1}{2}-2\delta}\Big)
   \Big(|\sigma|^\delta\langle\sigma\rangle^{-\frac{1}{2}-2\delta}\Big)
\in \Ltwo(\R^2).
$
Applying \eqref{eq:CS-etasigma} with
$G=(\eta+\sigma)\mathcal F^{-1}_{\eta,\sigma}[F_h]$
and with
$G=\mathcal F^{-1}_{\eta,\sigma}[\partial_\xi F_h]$,
we get
\begin{align}
\|\mathcal I_1(t)\|_{\Ltwo_\xi}
&\lesssim
t^{-1-\delta}
\big\|\langle\eta\rangle^{\frac{1}{2}+2\delta}\langle\sigma\rangle^{\frac{1}{2}+2\delta}
(\eta+\sigma)
\mathcal F^{-1}_{\eta,\sigma}[F_h]\big\|_{\Ltwo_{\eta,\sigma}\Ltwo_\xi}, \\
\|\mathcal I_2(t)\|_{\Ltwo_\xi}
&\lesssim
t^{-1-\delta}
\big\|\langle\eta\rangle^{\frac{1}{2}+2\delta}\langle\sigma\rangle^{\frac{1}{2}+2\delta}
\mathcal F^{-1}_{\eta,\sigma}[\partial_\xi F_h]
\big\|_{\Ltwo_{\eta,\sigma}\Ltwo_\xi}.
\end{align}
By Plancherel,
\begin{align}
\big\|\langle\eta\rangle^{\frac{1}{2}+2\delta}\langle\sigma\rangle^{\frac{1}{2}+2\delta}
(\eta+\sigma)
\mathcal F^{-1}_{\eta,\sigma}[F_h]\big\|_{\Ltwo_{\eta,\sigma}\Ltwo_\xi}+\big\|\langle\eta\rangle^{\frac{1}{2}+2\delta}\langle\sigma\rangle^{\frac{1}{2}+2\delta}
\mathcal F^{-1}_{\eta,\sigma}[\partial_\xi F_h]
\big\|_{\Ltwo_{\eta,\sigma}\Ltwo_\xi} \lesssim \|h\|_{H^{1+\frac{1}{2}+2\delta}_\xi}^3.
\end{align}
Therefore putting computations above together,
\[
\|\partial_\xi R[h](t)\|_{\Ltwo_\xi}
\lesssim
t^{-1-\delta}
\|h\|_{H^{2}_\xi}^3
\]
as claimed.\footnote{Here we do not really need $\delta$. Any positive number  is enough to ensure the integrablity in $t$. This corresponds to our earlier remark that it sufficient to take $W\in H^{3/2+}$.}
\end{proof}

Applying bounds above to the forcing term, we conclude the following:
\begin{proposition}\label{prop:forcing-bounds}
Let $v$ be defined by \eqref{eq:vdef}. Then for $t\ge2$,
\begin{equation}\label{eq:forcing-bounds}
\|R[v](t)\|_{\Linfty_\xi}+\|R[v](t)\|_{L^2_\xi}+\|\partial_\xi R[v](t)\|_{\Ltwo_\xi}
\le C t^{-1-\delta}\|v(t)\|_{\Htwo_\xi}^3
\le C t^{-1-\delta}\varepsilon_0^3(1+\log t)^6.
\end{equation}
Consequently, by Lemma \ref{lem:forcing},
$$
\|\reallywidehat{U(-t)\varepsilon}(t)\|_{\Linfty_\xi}+\|\reallywidehat{U(-t)\varepsilon}(t)\|_{L^2_\xi}+\|\partial_\xi\reallywidehat{U(-t)\varepsilon}(t)\|_{\Ltwo_\xi}
\le C t^{-1-\delta}\varepsilon_0^3(1+\log t)^6.
$$
\end{proposition}

\begin{proof}
Apply Lemma \ref{lem:Rlinf} and Lemma \ref{lem:dR} with $h=v(t,\cdot)$.
Since $v(t,\xi)=W(\xi)e^{-\ii\la|W(\xi)|^2\log t}$, repeated product/chain-rule bounds show
$$
\|v(t)\|_{\Htwo_\xi}\le C\|W\|_{\Htwo_\xi}(1+\log t)^2\leq C\varepsilon_0(1+\log t)^2.
$$
Insert into \eqref{eq:Rlinf} and \eqref{eq:dR}.
\end{proof}


\begin{proposition}\label{prop:Phi-eps}
Let $\Phi_\varepsilon$ be defined by
$$
\Phi_\varepsilon(t)=-\ii\int_t^\infty U(-s)\varepsilon(s)\,ds,
$$
where $\varepsilon$ is given by \eqref{eq:eps-def}.
Then for $T\ge2$,
\begin{equation}\label{eq:Phi-eps}
\|\Phi_\varepsilon\|_{X_T}\le C\varepsilon_0^3\,T^{\frac{\alpha-\delta}{2}}.
\end{equation}
\end{proposition}

\begin{proof}
By Lemma \ref{lem:forcing} and Proposition \ref{prop:forcing-bounds},
$$
\|\reallywidehat{U(-s)\varepsilon}(s)\|_{\Linfty_\xi}+\|\reallywidehat{U(-s)\varepsilon}(s)\|_{L^2_\xi}+\|\partial_\xi\reallywidehat{U(-s)\varepsilon}(s)\|_{\Ltwo_\xi}
\le C s^{-1-\delta}\varepsilon_0^3(1+\log s)^6.
$$
By definition of $\Phi_\varepsilon$,
$$
\wh{\Phi_\varepsilon}(t)= -\ii\int_t^\infty \reallywidehat{U(-s)\varepsilon}(s)\,ds.
$$
Therefore
$$
\|\wh{\Phi_\varepsilon}(t)\|_{L^\infty_\xi\cap L^2_\xi } \le \int_t^\infty \|\reallywidehat{U(-s)\varepsilon}(s)\|_{L^\infty_\xi\cap L^2_\xi }\,ds,\qquad
\|\partial_\xi\wh{\Phi_\varepsilon}(t)\|_{L^2_\xi}\le \int_t^\infty \|\partial_\xi\reallywidehat{U(-s)\varepsilon}(s)\|_{L^2_\xi}\,ds.
$$
Multiply by $t^\alpha$ in the above to conclude
$$
t^\alpha\int_t^\infty s^{-1-\delta}(1+\log s)^6\,ds\le C t^{\frac{\alpha-\delta}{2}}.
$$
Taking $\sup_{t\ge T}$ of the expression above gives \eqref{eq:Phi-eps}, noting
that $\alpha < \delta$.
\end{proof}


\section{Nonlinear estimates for the fixed point map}\label{sec:nonlinear}
In this section, we study  {\it a priori} estimates and difference estimates for $\Phi_{\mathrm{nl}}$ defined in \eqref{eq:Phi}.  Then we  prove that $\Phi$ is a contraction on a small ball in $X_T$
for $\varepsilon_0$ small and $T$ large.

We start with basic bounds for $g$ and $w=U(t)g$ from the $X_T$ norm. 

\begin{lemma}\label{lem:gL2}
If $g\in X_T$, then for all $t\ge T$,
\begin{equation}\label{eq:wLinf}
\|w(t)\|_{\Ltwo_x}\le C t^{-\alpha}\|g\|_{X_T},\qquad
\|w(t)\|_{\Linfty_x}\le C t^{-1/2-\alpha}\|g\|_{X_T}.
\end{equation}
\end{lemma}

\begin{proof}
Since $w(t)=U(t)g(t)$ and $U(t)$ is unitary on $\Ltwo$, $\|w(t)\|_{L_x^2}=\|g(t)\|_{L_x^2}$. The first inequality follows.

For the $L^\infty_x$ decay, apply Lemma \ref{lem:KPdisp} to $h=g(t)$:
$$
\|w(t)\|_{L_x^\infty}=\|U(t)g(t)\|_{L_x^\infty}\le C\Big(t^{-1/2}\|\wh g(t)\|_{L_\xi^\infty}+t^{-3/4}\|\partial_\xi\wh g(t)\|_{L_\xi^2}\Big)
\le C t^{-1/2-\alpha}\|g\|_{X_T}
$$as desired.
\end{proof}

We repeatedly use the identity, valid for any $a,b\in\C$,
\begin{equation}\label{eq:alg}
|a+b|^2(a+b)-|a|^2a
=2|a|^2b+a^2\overline{b}+2a|b|^2+\overline{a}b^2+|b|^2b.
\end{equation}
We will apply \eqref{eq:alg} with $a=u_{\mathrm{app}}$ and $b=w$.



\subsection{Estimate of $\|\wh{\Phi_{\mathrm{nl}}(g)}(t)\|_{\Linfty_\xi}$}
Fix $g\in B_M$ and write $w=U(\cdot)g$, $u=u_{\mathrm{app}}+w$.
Define the nonlinear difference
\begin{equation}\label{eq:Ndef}
\mathcal{N}(u,u_{\mathrm{app}})=|u|^2u-|u_{\mathrm{app}}|^2u_{\mathrm{app}}.
\end{equation}
Then
$$
\Phi_{\mathrm{nl}}(g)(t)=\ii\la\int_t^\infty U(-s)\mathcal{N}(u,u_{\mathrm{app}})(s)\,ds.
$$
Taking Fourier transforms gives
$$
\wh{\Phi_{\mathrm{nl}}(g)}(t,\xi)=\ii\la\int_t^\infty e^{\ii s\xi^2/2}\reallywidehat{\mathcal{N}(u,u_{\mathrm{app}})}(s,\xi)\,ds.
$$
Using \eqref{eq:alg} with $a=u_{\mathrm{app}}$ and $b=w$,
\begin{equation}\label{eq:Nexpand}
\mathcal{N}(u,u_{\mathrm{app}})=2|u_{\mathrm{app}}|^2w+u_{\mathrm{app}}^2\overline{w}
+2u_{\mathrm{app}}|w|^2+\overline{u_{\mathrm{app}}}w^2+|w|^2w.
\end{equation}
We note that
\begin{equation}
   \|g(t)\|_{H_x^{0,1}}\sim\|\wh g(t)\|_{H^1_\xi} \lesssim M t^{-\alpha} \log t, \qquad \|\varphi(t)\|_{H_x^{0,1}}\sim\|v(t)\|_{H^1_\xi} \lesssim \varepsilon_0\log t,
\end{equation}
\begin{equation}
\|\wh g(t)\|_{L^\infty_\xi} \lesssim  M t^{-\alpha}, \qquad \| v(t)\|_{L^\infty_\xi}= \|W\|_{L^\infty_\xi} \le \varepsilon_0.
\end{equation}
We focus on the analysis for $|u_{\mathrm{app}}|^2w$ and other terms can be estimated similarly.

By direct computations, using notations from Lemma \ref{lem:trilinear},
\begin{equation}
    e^{\ii s\xi^2/2}\reallywidehat{|u_{\mathrm{app}}|^2w}(s,\xi)=  \widehat{\mathcal{T}}(\varphi(s),\varphi(s),g(s))(s)= i \frac{1}{s}
|v|^2\widehat g (\xi)
+  R(\varphi(s),\varphi(s),g(s))(s,\xi)
\end{equation}
with 
\begin{equation}
    | R(\varphi(s),\varphi(s),g(s))(s,\xi)|
\lesssim
s^{-1-\delta}
\|\varphi\|_{H_x^{0,1}}^2\|g\|_{H_x^{0,1}}\lesssim \varepsilon_0^2 M s^{-1-\delta} s^{-\alpha} (\log s)^3 \lesssim \varepsilon_0^2 M s^{-1 - \alpha}.
\end{equation}
Clearly,  $\left|i \frac{1}{s}
|v|^2\widehat g (\xi)\right|\leq \varepsilon_0^2 M s^{-1 - \alpha}$. 
Therefore, with the same analysis applied to other terms, the leading order terms from $e^{\ii s\xi^2/2}\reallywidehat{\mathcal{N}(u,u_{\mathrm{app}})}(s,\xi)$, one has
$$
\|\wh{\Phi_{\mathrm{nl}}(g)}(t)\|_{L_\xi^\infty}
\le C\int_t^\infty \Big(\varepsilon_0^2 M s^{-1-\alpha}+\varepsilon_0 M^2 s^{-1-2\alpha}+M^3 s^{-1-3\alpha}\Big)\,ds.
$$
Multiply by $t^\alpha$ and take $\sup_{t\ge T}$. Since $\alpha>0$, the integrals converge and we get
\begin{equation}\label{eq:Phi-nl-Linf}
\sup_{t\ge T}t^\alpha\|\wh{\Phi_{\mathrm{nl}}(g)}(t)\|_{L_\xi^\infty}
\le C\Big(\varepsilon_0^2 M+\varepsilon_0 M^2 T^{-\alpha}+M^3T^{-2\alpha}\Big).
\end{equation}
In practice we choose   $T$ large so that the latter terms are small; the leading coefficient is $\varepsilon_0^2M$.

\subsection{Lipschitz bound in $\Linfty_\xi$}
Let $g_1,g_2\in B_M$, with associated $w_j=U(\cdot)g_j$ and $u_j=u_{\mathrm{app}}+w_j$.
We estimate
\begin{equation}
\|\wh{\Phi_{\mathrm{nl}}(g_1)}(t)-\wh{\Phi_{\mathrm{nl}}(g_2)}(t)\|_{L_\xi^\infty}
\le C\int_t^\infty \|\reallywidehat{ U(-s)\Big(\mathcal{N}(u_1,u_{\mathrm{app}})-\mathcal{N}(u_2,u_{\mathrm{app}})\Big)}\|_{L^\infty_\xi}\,ds.  
\end{equation}
Using the algebraic identity for the cubic map, one can write for $F(z)=|z|^2z$,
$$
F(u_1)-F(u_2)=Q(u_1,u_2)(u_1-u_2)
$$
for some quadratic polynomial $Q(u_1,u_2)$.

Now we can apply Lemma \ref{lem:trilinear} and the same analysis above to conclude that
\begin{equation}\label{eq:Lip-Linf-final}
\sup_{t\ge T}t^\alpha\|\wh{\Phi_{\mathrm{nl}}(g_1)}(t)-\wh{\Phi_{\mathrm{nl}}(g_2)}(t)\|_{L_\xi^\infty}
\le C\Big(\varepsilon_0^2+\varepsilon_0MT^{-\alpha}+M^2T^{-2\alpha}\Big)\|g_1-g_2\|_{X_T}.
\end{equation}
\subsection{Estimate of $\|\partial_\xi \wh{\Phi_{\mathrm{nl}}(g)}(t)\|_{L^2_\xi}$}
We control $\partial_\xi$ directly on the Fourier side,
in the same spirit as the computation in \cite{KP}. 
Write, schematically,
$$
U(-s)(|u(s)|^2u(s)) = e^{-\ii s\partial_{xx}/2}\big(|u(s)|^2u(s)\big),
\qquad
u(s)=e^{\ii s\partial_{xx}/2}f(s).
$$
Then
$$
U(-s)(|u|^2u)=e^{-\ii s\partial_{xx}/2}\Big(\big(e^{\ii s\partial_{xx}/2}f\big)\big(e^{-\ii s\partial_{xx}/2}\overline{f}\big)\big(e^{\ii s\partial_{xx}/2}f\big)\Big).
$$
Differentiate in $\xi$ on the Fourier side following \cite{KP}, one has
\begin{align}
\partial_\xi \iint e^{is\eta\sigma}\,
&\widehat f(s,\xi-\eta)\,\widehat{\overline f}(s,\xi-\sigma-\eta)\,\widehat f(s,\xi-\sigma)\,
d\eta\,d\sigma
\\&=\iint e^{is\eta\sigma}\,
\partial_\xi\widehat f(s,\xi-\eta)\,\widehat{\overline f}(s,\xi-\sigma-\eta)\,\widehat f(s,\xi-\sigma)\,
d\eta\,d\sigma  \\
&\qquad +\ \text{similar terms}.
\end{align}
where ``similar terms'' represents those terms where $\partial_\xi$ hits the Fourier transform of
the other profiles. These ``similar terms'' can be
estimated in the same way as the first one above. Then, after redistributing the
phase on the three profiles, we see that
\begin{equation}\label{eq:dxiC-bound}
\|\partial_\xi\reallywidehat{U(-s)(|u|^2u)}\|_{L^2_\xi}\le C\|u(s)\|_{L^\infty_x}^2\|\partial_\xi\wh f(s)\|_{L^2_\xi}.
\end{equation}
Note that the nonlinear term in $\Phi_{\mathrm{nl}}$ is
$$
U(-s)\Big(|u|^2u-|u_{\mathrm{app}}|^2u_{\mathrm{app}}\Big)=U(-s)\Big(2|u_{\mathrm{app}}|^2w+u_{\mathrm{app}}^2\overline{w}
+2u_{\mathrm{app}}|w|^2+\overline{u_{\mathrm{app}}}w^2+|w|^2w\Big).
$$
Again, we focus on $|u_{\mathrm{app}}|^2w$ and other terms will follow similarly. 

Applying \eqref{eq:dxiC-bound} to each cubic term and using the triangle inequality yields
\begin{align}\label{eq:dxi-nl-diff}
\|\partial_\xi \reallywidehat{U(-s)|u_{\mathrm{app}}|^2w}\|_{L^2_\xi}
&\le C\Big(\|u_{\mathrm{app}}(s)\|_{L^\infty_x}^2\|\partial_\xi\wh g(s)\|_{L^2_\xi}+\|u_{\mathrm{app}}(s)\|_{L^\infty_x}\|w(s)\|_{L^\infty_x}\|\partial_\xi v(s)\|_{L^2_\xi}\Big)\\
&\le C(\varepsilon_0^2M s^{-1-\alpha} \log s + \varepsilon_0 M^2 s^{-1 - 2\alpha}\log s).
\end{align}
%
Taking $\partial_\xi$ of $\wh{\Phi_{\mathrm{nl}}(g)}(t)$, one has
$$
\|\partial_\xi\wh{\Phi_{\mathrm{nl}}(g)}(t)\|_{L^2_\xi}\le C\int_t^\infty \|\partial_\xi\reallywidehat{U(-s)\mathcal{N}(u,u_{\mathrm{app}})}(s)\|_{L^2_\xi}\,ds.
$$
Using \eqref{eq:dxi-nl-diff},
after multiplying by $t^\alpha (1+\log t)^{-1}$ and taking $\sup_{t\ge T}$, the time integrals are handled exactly as in the $\Linfty_\xi$ Lipschitz argument.
This yields the bound
\begin{equation}\label{eq:Phi-nl-dxi}
\sup_{t\ge T}\Big(t^\alpha (1+\log t)^{-1}\Big)\|\partial_\xi\wh{\Phi_{\mathrm{nl}}(g)}(t)\|_{L^2_\xi}
\le C\Big(\varepsilon_0^2 M+\varepsilon_0M^2T^{-\alpha}+M^3T^{-2\alpha}\Big).
\end{equation}

\subsection{Lipschitz bound in $\partial_\xi\Ltwo_\xi$}
Let $g_1,g_2\in B_M$, $u_j=u_{\mathrm{app}}+U(\cdot)g_j=u_{\mathrm{app}}+w_j$, and $f_j=\varphi+g_j$.
We estimate
$$
\|\partial_\xi\wh{\Phi_{\mathrm{nl}}(g_1)}(t)-\partial_\xi\wh{\Phi_{\mathrm{nl}}(g_2)}(t)\|_{L^2_\xi}
\le C\int_t^\infty \|\partial_\xi\reallywidehat{U(-s)\big(F(u_1)-F(u_2)\big)}\|_{L^2_\xi}\,ds,
$$
where $F(z)=|z|^2z$. Using the expansion \eqref{eq:Nexpand} applied to $w_j$ and  taking  the difference, the analysis follows by the same argument as \eqref{eq:dxi-nl-diff} with the following bounds:
\begin{equation}\label{eq:diffbounds}
\|\partial_\xi(\wh g_1-\wh g_2)(s)\|_{L^2_\xi}\le s^{-\alpha}(\log s)\|g_1-g_2\|_{X_T},\,\,\|w_1(s)-w_2(s)\|_{L^\infty_x}\le Cs^{-1/2-\alpha}\|g_1-g_2\|_{X_T}.  
\end{equation}
We conclude
\begin{equation}\label{eq:Lip-dxi-final}
\sup_{t\ge T}t^\alpha(1+\log t)^{-1}\|\partial_\xi\wh{\Phi_{\mathrm{nl}}(g_1)}(t)-\partial_\xi\wh{\Phi_{\mathrm{nl}}(g_2)}(t)\|_{L^2_\xi}
\le C\Big(\varepsilon_0^2+\varepsilon_0 MT^{-\alpha}+M^2T^{-2\alpha}\Big)\|g_1-g_2\|_{X_T}.
\end{equation}
\subsection{Estimate of $\|\wh{\Phi_{\mathrm{nl}}(g)}(t)\|_{L^2_\xi}$}
Using Lemma \ref{lem:gL2} and the expansion \eqref{eq:Nexpand}, we directly have
\begin{align}
\left\Vert\wh{\Phi_{\mathrm{nl}}(g)}(t,\xi)\right\Vert_{L^2_\xi}&\leq\int_t^\infty \left\Vert\mathcal{N}(u,u_{\mathrm{app}})(s,x)\right\Vert_{L^2_x}\,ds\\
&\leq C (\varepsilon_0^2M t^{-\alpha} + \varepsilon_0M^2 t^{-2\alpha} + M^3 t^{-3\alpha}). 
\end{align}
Therefore, 
\begin{equation}\label{eq:Phi-nl-2}
\sup_{t\ge T}t^\alpha\|\wh{\Phi_{\mathrm{nl}}(g)}(t)\|_{L_\xi^2}
\le C\Big(\varepsilon_0^2 M+\varepsilon_0 M^2 T^{-\alpha}+M^3T^{-2\alpha}\Big).
\end{equation}
\subsection{Lipschitz bound in $L^2_\xi$} Again, this is a direct consequence of the expansion \eqref{eq:Nexpand} and the decay estimate from \eqref{eq:diffbounds}. We directly conclude that
\begin{equation}\label{eq:Lip-2-final}
\sup_{t\ge T}t^\alpha\|\wh{\Phi_{\mathrm{nl}}(g_1)}(t)-\wh{\Phi_{\mathrm{nl}}(g_2)}(t)\|_{L_\xi^2}
\le C\Big(\varepsilon_0^2+\varepsilon_0MT^{-\alpha}+M^2T^{-2\alpha}\Big)\|g_1-g_2\|_{X_T}.
\end{equation}
\subsection{Conclusion: contraction and proof of Theorem \ref{thm:main}}
Combine the $\Linfty_\xi$ Lipschitz estimate \eqref{eq:Lip-Linf-final}, the $\partial_\xi\Ltwo_\xi$ Lipschitz estimate \eqref{eq:Lip-dxi-final}, and the Lipschitz estimate in $L^2_\xi$ \eqref{eq:Lip-2-final}  to obtain
\begin{equation}\label{eq:PhiLip}
\|\Phi_{\mathrm{nl}}(g_1)-\Phi_{\mathrm{nl}}(g_2)\|_{X_T}
\le C\Big(\varepsilon_0^2+\varepsilon_0MT^{-\alpha}+M^2T^{-2\alpha}\Big)\|g_1-g_2\|_{X_T}.
\end{equation}
Next, $\Phi_\varepsilon$ satisfies \eqref{eq:Phi-eps}, hence is $T$-small since $\alpha<\delta$.
Finally, the self-mapping bounds \eqref{eq:Phi-nl-Linf},  \eqref{eq:Phi-nl-dxi}, and \eqref{eq:Phi-nl-2} show that $\Phi$ maps $B_M$ to itself for $M\sim \varepsilon_0$ and $T$ large.

Choose $\varepsilon_0$ so small that $C\varepsilon_0^2\le 1/8$.
Then choose $T$ large so that
$$
C\Big(\varepsilon_0^2+\varepsilon_0MT^{-\alpha}+M^2T^{-2\alpha}\Big)\le \frac{1}{2}.
$$
With these choices, \eqref{eq:PhiLip} implies $\Phi$ is a contraction on $B_M$ with contraction constant $\le 1/2$,
and $\Phi(B_M)\subset B_M$. By Banach's fixed point theorem, there exists a unique $g\in B_M$ with $g=\Phi(g)$.
Then $u=u_{\mathrm{app}}+U(t)g$ solves NLS on $[T,\infty)$ and satisfies \eqref{eq:mainbound}.
This proves \eqref{eq:mainbound}.

{\bf Proof of \eqref{eq:u-asymp-main}.}
We start from the identity
$$
u(t)=U(t)f(t),
\qquad
\wh f(t)=v(t)+\wh g(t).
$$
Applying \eqref{eq:SP-template} to $h=f(t)$, we get
\begin{equation}\label{eq:Utf-main}
u(t,x)=U(t)f(t,x)
=(2\pi)^{-1/2}(\ii t)^{-1/2}e^{\ii x^2/(2t)}\,\wh f\!\Big(t,\frac{x}{t}\Big)+\mathrm{Rem}_1(t,x),
\end{equation}
where
\begin{equation}\label{eq:Rem1}
\|\mathrm{Rem}_1(t)\|_{L^\infty_x}\le C\,t^{-3/4}\|\partial_\xi \wh f(t)\|_{L^2_\xi}.
\end{equation}
Since $\partial_\xi \wh f=\partial_\xi v+\partial_\xi \wh g$, we have
$$
\|\partial_\xi \wh f(t)\|_{L^2_\xi}\le \|\partial_\xi v(t)\|_{L^2_\xi}+\|\partial_\xi \wh g(t)\|_{L^2_\xi}.
$$
From the explicit formula $v(t,\xi)=W(\xi)e^{-\ii\la|W(\xi)|^2\log t}$ and the smallness assumption
$\|W\|_{H^2_\xi}\le \varepsilon_0$, it follows by the product and chain rules that
$$
\|\partial_\xi v(t)\|_{L^2_\xi}\le C\,\varepsilon_0(1+\log t),
$$
while \eqref{eq:mainbound} gives $\|\partial_\xi \wh g(t)\|_{L^2_\xi}\le Ct^{-\alpha}(1 + \log t)$.
Hence \eqref{eq:Rem1} yields
\begin{equation}\label{eq:Rem1-final}
\|\mathrm{Rem}_1(t)\|_{L^\infty_x}\le C\,t^{-3/4}\varepsilon_0(1+\log t)+C\,t^{-3/4-\alpha}(1 + \log t).
\end{equation}
Write
$$
\wh f\Big(t,\frac{x}{t}\Big)=v\!\Big(t,\frac{x}{t}\Big)+\wh g\!\Big(t,\frac{x}{t}\Big),
$$
and define
$$
\mathrm{Rem}_2(t,x)=(2\pi)^{-1/2}(\ii t)^{-1/2}e^{\ii x^2/(2t)}\,\wh g\!\Big(t,\frac{x}{t}\Big).
$$
Then
$$
\|\mathrm{Rem}_2(t)\|_{L^\infty_x}\le (2\pi)^{-1/2}t^{-1/2}\|\wh g(t)\|_{L^\infty_\xi}\le C\,t^{-1/2-\alpha}
$$
by \eqref{eq:mainbound}. Combining \eqref{eq:Utf-main} with the explicit expression
\begin{equation}\label{eq:v-on-ray}
v\!\Big(t,\frac{x}{t}\Big)=W\!\Big(\frac{x}{t}\Big)\exp\!\Big(-\ii\la\Big|W\!\Big(\frac{x}{t}\Big)\Big|^2\log t\Big),
\end{equation}
we obtain \eqref{eq:u-asymp-main} with
$
\mathrm{Err}(t,x)=\mathrm{Rem}_1(t,x)+\mathrm{Rem}_2(t,x).
$
The bound \eqref{eq:u-asymp-err} follows from \eqref{eq:Rem1-final} and the estimate on $\mathrm{Rem}_2$.



\begin{thebibliography}{99}
	\bibitem{CW}
		Cazenave, T. and Weissler, B. 
		The Cauchy problem
		for the critical nonlinear Schr\"odinger equation in $H^{s}$. \emph{Nonlinear
			Anal}. 14 (1990), no. 10, 807\textendash 836.
\bibitem{CP}
Chen, G. and Pusateri, F., . The 1-dimensional nonlinear Schrödinger equation with a weighted $L^1$ potential. Analysis \& PDE, 15, 2022 , no. 4.  937--982.
		\bibitem{DZ}
		Deift, P. and Zhou, X. 
		Long-time asymptotics for
		solutions of the NLS equation with initial data in a weighted Sobolev
		space. Dedicated to the memory of Jurgen K. Moser.
		\emph{Comm. Pure Appl. Math.} 56 (2003), no. 8, 1029\textendash 1077.
		
		

\bibitem{GPR} 
Germain, P., Pusateri, F. and Rousset, F. 
The nonlinear Schr\"odinger equation with a potential.
\emph{Ann. Inst. H. Poincaré Anal. Non Linéaire} 35 (2018), no. 6, 1477\textendash 1530. 

\bibitem{HN} 
		Hayashi, N. and Naumkin, P.I.
		Asymptotics for large time of solutions to the nonlinear Schr\"odinger and Hartree equations.
		\emph{Amer. J. Math.} 120 (1998), 369\textendash 389.


\bibitem{HN2} Hayashi, N. and Naumkin, P.I. Domain and range of the modified wave operator for Schrödinger equations with a critical nonlinearity. \emph{Communications in mathematical physics} 267, no. 2 (2006): 477-492.        
		
		\bibitem{IT} 
		Ifrim, M. and Tataru, D. 
		Global bounds for the cubic nonlinear Schr\"odinger equation (NLS) in one space dimension. 
		\emph{Nonlinearity} 28 (2015), no. 8, 2661--2675.
	
		\bibitem{KP}
		Kato, J. and Pusateri, F.
		A new proof of long-range
		scattering for critical nonlinear Schr\"odinger equations. \emph{Differential
			Integral Equations}, 24 (2011), 923\textendash 940.

    \bibitem{KM} Kawamoto M, Mizutani H. Modified wave operators for the defocusing cubic nonlinear Schr\" odinger equation in one space dimension with large scattering data. arXiv preprint arXiv:2506.01871. 2025 Jun 2.        
 \bibitem{LS} 
		Lindblad, H. and Soffer, A.
		Scattering and small data completeness for the critical nonlinear Schr\"odinger equation. \emph{Nonlinearity} 19 (2006), no. 2, 345–353.           
\bibitem{MurphyReview}
Murphy, J.
A review of modified scattering for the 1D cubic NLS,
\emph{RIMS K\^oky\^uroku Bessatsu} (2021).	
	\bibitem{O} Ozawa, T. Long range scattering for nonlinear Schr\"odinger equations in one space dimension. \emph{Comm.
	Math. Phys.} 139 (1991), no. 3, 479--493.
  \bibitem{Se} Segata, J-I,, Final state problem for the cubic nonlinear Schrödinger equation with repulsive delta potential. \emph{Comm. Partial Differential Equations} 40 (2015), no. 2, 309–328..          

\end{thebibliography}
\end{document}